\documentclass{amsart}

\usepackage[letterpaper,top=2cm,bottom=2cm,left=3cm,right=3cm,marginparwidth=1.75cm]{geometry}

\usepackage{latexsym,amsmath,a4wide,amssymb,amsthm,graphics,enumitem}
\usepackage[colorlinks=true, allcolors=blue]{hyperref}

\newtheorem{theorem}{Theorem}[section]
\newtheorem{example}{Example}[section]

\newtheorem{proposition}{Proposition}[section]
\newtheorem{corollary}{Corollary}[section]
\newtheorem{remark}{Remark}[section]

\newcommand{\comp}{\mathbb{C}}
\newcommand{\RR}{\mathbb{R}}
\newcommand{\CC}{\mathbb{C}}

\renewcommand{\leq}{\leqslant}
\renewcommand{\geq}{\geqslant}

\usepackage{color}

\title{Matrix pencils with the numerical range equal to the whole complex plane}
\author{Vadym Koval, Patryk Pagacz}

\begin{document}
\maketitle

\begin{abstract}
The main purpose of this article is to show that the numerical range of a linear pencil $\lambda A+B$ is equal to $\comp$ if and only if $0$ belongs to the convex hull of the joint numerical range of $A$ and $B$. We also prove that if  the numerical range of a linear pencil $\lambda A+B$ is equal to $\comp$ and $A+A^*,B+B^*\geq 0$, then $A$ and $B$ have a common isotropic vector.

Moreover, we improve the classical result which describes Hermitian linear pencils.

\end{abstract}








\section{Introduction}
Let $\comp^{N\times N}$ and $\mathcal{H}(N)$ denote the algebra of $N\times N$ complex matrices and the algebra of $N\times N$ Hermitian matrices, respectively.
By \emph{a matrix polynomial} $P(\lambda)$ we mean, a polynomial with coefficients from $\comp^{N\times N}$, i.e.
$$P(\lambda)=A_k\lambda^k+ A_{k-1}\lambda^{k-1}+\dots+A_0,\quad A_0,A_1,\dots,A_k \in \comp^{N\times N}.$$
For $k=1$, the matrix polynomial is called \emph{a linear pencil}.
In this paper we deal with \emph{a numerical range of matrix polynomial}, i.e. $$W(P(\lambda)):=\{\lambda\in \comp : x^*P(\lambda)x=0, \textnormal{ for some nonzero } x\in \comp^N\}.$$
It is easy to observe that $W(\lambda I-A)$ is equal to \emph{a numerical range of matrix} $A$, i.e. $W(A)=\{x^*Ax: x^*x=1\}$. More basic properties of the numerical ranges of matrix polynomials one can find in \cite{Li-Rodman}. 

The notion of numerical range of a matrix can be extended for $m$-tuple of matrices. We define
$W(A_1,A_2,\ldots,A_m):=\{(x^*A_1x,x^*A_2x,\dots,x^*A_mx): x^*x=1\}$ and call it \emph{a joint numerical range}.
If it is convenient we will identify $\comp$ with $\mathbb{R}^2$. In particular, we identify $W(A,B)$ with $W(A_1,A_2,B_1,B_2)$, where $A=A_1+iA_2, B=B_1+iB_2$ and $A_1,A_2,B_1,B_2$ are Hermitian (self-adjoint) matrices.

Many matrix polynomials find their role in applications (see e.g. \cite{NLEVP,MR3396732}). The study of numerical ranges of matrix polynomials has a crucial role in the stability theory. In particular, the system represented by the matrix polynomial is stable,  if its numerical range is contained in the left half-plane. Thus the shape of a numerical range was the topic of many papers (see e.g. \cite{ITO2016683,CHIEN200269,NAKAZATO2001105,PSARRAKOS2000127,Bebiano,Ratio,Beb1,CHIEN2002205,MAROULAS1997101,MAROULAS199897,MAROULAS199641,approxBoundary,SzymanskiWojtylak,mehl2021matrix}).
Moreover, the classical approach to investigating matrix polynomials is linearization, where the polynomial is converted into a larger matrix pencil with the similar properties (see \cite{Linearizations,Linearizations2,MehrmannWatkins,KRUPNIK199845,LinLancaster}). Thus, it is most natural to consider linear pencils, so we mostly do that in this paper. 

For matrices $A_0,A_1,\dots,A_m\in \comp^{N\times N}$ a nonzero vector  $v\in \comp^N$ is called \emph{a common isotropic vector} if
$$v^*A_0v=v^*A_1v=\dots =v^*A_mv=0.$$
It is easy to see that if the matrices $A_0,A_1,\dots,A_k\in \comp^{N\times N}$ have a common isotropic vector, then $W(P(\lambda))=\comp$, where
$P(\lambda)=A_k\lambda^k+ A_{k-1}\lambda^{k-1}+\dots+A_0$.
The Example \ref{VadymEx} shows that the converse implication for quadratic polynomials is not true, even if all matrices are Hermitian. On the other side one can deduce from Theorem 4.1 \cite{Li-Rodman} that two Hermitian matrices $A,B\in \comp^{N\times N}$ have a common isotropic vector if and only if $W(A\lambda+B)=\comp$.

Unfortunately, there is an error in Theorem 4.1(e) \cite{Li-Rodman}.
It was observed by the authors of \cite{Bebiano}, but they also made an error in their improvement.

The main purpose of this paper is to adjudicate when the equality $W(A\lambda+B)=\comp$ implies that $A$ and $B$ have a common isotropic vector.

The paper is divided into tree parts. 
In Section \ref{Main} we describe linear pencils $P(\lambda)$ such that $W(P(\lambda))=\comp$. 
Moreover, we solve an open question proposed in \cite{mehl2021matrix} and inspired by the study on dissipative Hamiltonian matrix pencils (see also \cite{MMWsimax2018,SzymanskiWojtylak,MMW2021}).
Namely, we prove that for linear pencils $P(\lambda)=A\lambda+B$ such that $ A+A^*, B+B^* \geq 0$ and $W(P(\lambda))=\comp$, the matrices $A, B$ have a common isotropic vector. Our Theorem \ref{openMMW} extends a partial result Theorem 5.7 \cite{mehl2021matrix}.

In Section \ref{self-a} we consider matrix polynomials with Hermitian coefficients. We improve the erroneous Theorem 4.1 from \cite{Li-Rodman} (and also Theorem 1 from \cite{Bebiano}). In particular, we show that for linear pencils with Hermitian coefficients the equality $W(P(\lambda))=\mathbb{R}$ is never satisfied. 

Additionally in Section \ref{last}, we show that second-degree matrix polynomial $P(\lambda)$ with semidefinite coefficients satisfy $W(P(\lambda))=\comp$ if and only if matrix coefficients have a common isotropic vector.
Moreover, we give an example (Example \ref{VadymEx}) of three $3\times 3$ Hermitian matrices $A,B,C$ without common isotropic vector such that $W(P(\lambda))=\comp$, where $P(\lambda)=A\lambda^2+B\lambda+C$.

\section{Linear Pencils with $W(P(\lambda))=\comp$}\label{Main}

The well-known Toeplitz-Hausdorff theorem(see e.g. \cite{davis_1971}) says that the numerical range of a matrix is always convex.
This result was extended in \cite{extHT}, namely the joint numerical range of three Hermitian matrices of dimension larger than $2$ is convex. But the joint numerical range of four Hermitian matrices often is not convex (see \cite{au1979remark}). In other words,
the joint numerical range of two (arbitrary) matrices does not have to be convex. 

In Remark 1 \cite{PSARRAKOS2000127} the authors observed that the numerical range of a linear pencil with one Hermitian matrix is equal $\comp$ if $0$ belongs to the convex hull of the joint numerical range of the linear pencil's coefficients. 
In view of results from \cite{extHT} the only example with nonconvex joint numerical range can be found for $2\times 2$ matrices (see Corollary \ref{3hermitian}).

\begin{example}\label{ex2x2}
Let $A=\begin{bmatrix}
    0 & 0 \\
    2 & 0
\end{bmatrix}$
and $B=\begin{bmatrix}
     1 & 0 \\
     0 & -1
\end{bmatrix}$.
Then the pencil $P(\lambda)=A\lambda+B$ has no common isotropic vector, but $W(P(\lambda))=\mathbb{C}$.
\end{example}

\begin{proof}
Indeed, if we take a vector $v=(re^{-i\phi},1)$, then $W(P(\lambda))\ni-\frac{v^*Bv}{v^*Av}=\frac{1-r^2}{2r}\cdot e^{i\phi}$. So $W(P(\lambda))=\mathbb{C}$.

On the other hand, for $v=(x,y)$ equality $v^*Bv=0$ holds if and only if $|x|=|y|$ and equality $|v^*Av|=0$ if and only if $|x||y|=0$.
\end{proof}

Fortunately, Remark 1 from \cite{PSARRAKOS2000127} can be extended without any assumptions about matrices.

\begin{theorem}\label{convIso}
Let $A,B\in \comp^{N\times N}$. Then $W(A\lambda+B)=\comp$ if and only if $0$ belongs to the convex hull of $ W(A,B)$.
\end{theorem}
\begin{proof}
 Let us denote $P(\lambda)=A\lambda+B$. The matrices $A,B$ have a common isotropic vector if and only if $0\in W(A,B)$. Thus without loss of generality we can assume that $A,B$ do not have a common isotropic vector. 

($\Leftarrow$) Let us assume that $0$ belongs to the convex hull of $ W(A,B)$.

 Let $A=A_1+i A_2$ and $B=B_1+i B_2$, where $A_1,A_2,B_1,B_2\in \mathcal{H}(N)$.
There are points $x,y\in W(A,B)$ such that $tx+(1-t)y=0\not\in W(A,B)$, for some $t\in (0,1)$.

\textbf{Case $N>2$:}
\newline Let us consider an arbitrary three dimensional projection $Q:\mathbb{R}^4\to\mathbb{R}^4$ ($\dim \ker Q=1$). 

By \cite{extHT}  a joint numerical range of three Hermitian operators is convex. Thus it is easy to see that the set $QW(A_1,A_2,B_1,B_2)$ is convex as well.
Hence $0=Q0\in Q[x,y]=[Qx,Qy]\subset QW(A_1,A_2,B_1,B_2)$.

Hence there exists a point $0\not=w\in W(A,B)$ such that $Qw=0$. In other words, the line $\ker Q$ has a joint point with $W(A,B)$. 
Since $Q$ was arbitrary, all one dimensional subspaces (lines) intersect the joint numerical range $W(A,B)$. In particular, for any $\lambda\in \comp$ the two dimensional subspace  $\{(z,w)\in \comp^2: z\lambda+w=0\}$ intersect $W(A,B)$. Thus $W(P(\lambda))=\comp$. 

\textbf{Case $N=2$:}
\newline It is well-known that the joint numerical range of three $2\times 2$ Hermitian matrices is an ellipsoid (may be a degenerated one). Thus the set $QW(A_1,A_2,B_1,B_2)$ is an ellipsoid for any three dimensional projection $Q:\mathbb{R}^4\to\mathbb{R}^4$.
Moreover, $0\in \text{conv}(QW(A_1,A_2,B_1,B_2))$, since $0\in \text{conv}(W(A_1,A_2,B_1,B_2))$.
Thus $QW(A_1,A_2,B_1,B_2)\cap QL_\lambda\not=\emptyset$, where $L_\lambda=\{(z,w)\in \comp^2: z\lambda=w\}$, as any line passing through the point in the convex hull of an ellipsoid intersects it.
Hence $W(A,B)\cap L_\lambda\not=\emptyset$, as one can chose $Q$ to be a projection parallel to a line $l$ such that $l\in L_{\lambda}$. Therefore, $W(P(\lambda))=\comp$.

($\Rightarrow$) Now, let us assume that $0$ does not belong to the convex hull of $W(A,B)$.

Since $0 \not\in \text{conv}(W(A,B))$, the set $W (A,B)$ can be separated from $0$ by some hyper-plane, in other words there is a linear functional $f : \mathbb{R}^4 \ni (x_1,x_2,x_3,x_4) \mapsto \sum_{i=1}^4 \alpha_i x_i\in \mathbb{R}$ and $\varepsilon>0$ such that $f(W(A,B))>2\varepsilon$. 

Let us take (we can assume that $\alpha_1\not=0$) $\Tilde{\lambda}=\lambda_1+i\lambda_2\in \comp$ such that $$\alpha_1+\alpha_3\lambda_1+\alpha_4\lambda_2=\frac1k$$
and 
$$ \alpha_2-\alpha_3\lambda_2+\alpha_4\lambda_1 = \frac1k.$$

Let us consider an intersection of the sets $\{(z,w)\in \comp^2: \Tilde{\lambda} z=w\}=\{(z_1,z_2,w_1,w_2)\in \mathbb{R}^4 : w_1=\lambda_1z_1-\lambda_2z_2, w_2=\lambda_1z_2+\lambda_2z_1 \}$ and $f^{-1}(\varepsilon)$.
By simple computations one can get that this intersection is equal 
$$I:=\{(z_1,z_2,\lambda_1z_1-\lambda_2z_2,\lambda_1z_2+\lambda_2z_1): z_2=-z_1+k\varepsilon \}.$$
Thus for any ball there is large enough $k$ such that $I$ has  empty intersection with that ball. 

Let $\Omega:=\{x\in f^{-1}(\varepsilon): tx\in W(A,B) \textnormal{ for some } t>0 \}$. Since $W(A,B)$ is compact, $\Omega$ is compact as well. So for some $k$ one can achieve $\Omega \cap I = \emptyset$. And as a consequence,  $\{(z,w)\in \comp^2: \Tilde\lambda z=w\}\cap W(A,B)=\emptyset$. Hence $-\Tilde{\lambda} \not\in W(P(\lambda)) $.
\end{proof}

\begin{remark}
The joint numerical range of two matrices does not have to be convex (see \cite{au1979remark}). In fact, for any three independent with $Id$ Hermitian matrices $A_1,A_2,B_1$ there exists a Hermitian matrix $B_2$ such that $W(A_1,A_2,B_1,B_2)=W(A_1+iA_2,B_1+iB_2)$ is not convex (see \cite{LiPoon1999}).
\end{remark}

The following corollary shows that Example \ref{ex2x2} was not possible for larger matrices.

\begin{corollary}\label{3hermitian}
Let $A\in\comp^{N\times N}$ and $B\in\comp^{N\times N}$ be Hermitian matrices, where $N>2$. If $\mathbb{C}= W(A\lambda+B)$, then $A$ and $B$ have a common isotropic vector.
\end{corollary}
\begin{proof}
Without loss of generality let us assume that $B\in \mathcal{H}(N)$ and $A=A_1+iA_2$, where $A_1,A_2\in \mathcal{H}(N)$. Thus $W(A,B)=W(A_1,A_2,B)\subset \mathbb{R}^3\times\{0\}$ and by \cite{extHT} it is a convex set. Thus the thesis follows from Theorem \ref{convIso}.
\end{proof}

If we consider a linear pencil with two Hermitian coefficients, we can strengthen the above thesis.

\begin{corollary}\label{CorForBebiano}
Let $A\in\comp^{N\times N}$ and $B\in\comp^{N\times N}$ be Hermitian matrices. If $\mathbb{R}\subset W(A\lambda+B)$, then $A$ and $B$ have a common isotropic vector.
\end{corollary}
\begin{proof}
Let $A$ and $B$ do not have a common isotropic vector.
Since $W(A,B)=W(A+iB)\subset\mathbb{R}^2$ is convex, then by Theorem \ref{convIso} $0\notin W(A,B)$. Thus there is $\tilde\lambda\in\mathbb{R}$ such that $L_{\tilde\lambda}\cap W(A,B)=\emptyset$, where $L_{\tilde\lambda}=\{x+iy\in \comp : \tilde\lambda x=y\}$. Therefore, $-\tilde\lambda \not\in W(A,B)$.
\end{proof}

In Theorem 5.7 from \cite{mehl2021matrix} the authors showed that if some three of the four matrices $R_1,R_2,J_1,J_2$, such that $$R_1\geq 0,\quad R_2\geq 0, \quad J_1^*=-J_1, \quad J_2^*=-J_2$$
do not have a common isotropic vector, then $W(R_1+J_1,R_2+J_2)\not=\comp$.

Now, we prove that it is enough to assume the four matrices $R_1,R_2,J_1,J_2$ do not have a common isotropic vector. And therefore we solve an open question proposed in \cite{mehl2021matrix}.

\begin{theorem}\label{openMMW}
Let $N>2$ and $P(\lambda)=A\lambda+B$ be a pencil such that $A=R_1+J_1\in \comp^{N\times N}$ and $B=R_2+J_2\in \comp^{n\times n}$, where $R_i\geq 0$, $J^*_i=-J_i$, for $i=1,2$ and $W(P(\lambda))=\mathbb{C}$ then $A,B$ have a common isotropic vector.
\end{theorem}
\begin{proof}
The joint numerical range $W(A,B)$ is naturally identified with $W(R_1,-iJ_1,R_2,-iJ_2)$.

Let $Q :\mathbb{R}^4\to\mathbb{R}^4$ be an orthogonal projection onto the space $lin\{(1,0,1,0),(0,1,0,0),(0,0,0,1)\}$. In other words, $Q(x_1,x_2,x_3,x_4)=(\frac{x_1+x_3}{\sqrt{2}},x_2,x_4)$, where the space $Q\mathbb{R}^4$ comes equipped with orthonormal basis  $\{(\frac{\sqrt{2}}{2},0,\frac{\sqrt{2}}{2},0),(0,1,0,0),(0,0,0,1)\}$.

Due to \cite{extHT} the  joint numerical range $W(\frac{R_1+R_2}{\sqrt{2}},-iJ_1,-iJ_2)$ is convex. Thus the  set  $QW(A,B)=QW(R_1,-iJ_1,R_2,-iJ_2)=W(\frac{R_1+R_2}{\sqrt{2}},-iJ_1,\frac{R_1+R_2}{\sqrt{2}},-iJ_2)$ is convex as well.

Let us assume for a while that $0\notin QW(A,B)$. Thus there is a linear functional $f:Q\mathbb{R}^4\ni (x_1,x_2,x_3) \mapsto \alpha_1x_1+\alpha_2x_2+\alpha_3x_3 \in \mathbb{R}$ such that $f(QW(A,B))>0$. Since $QW(A,B)$ is compact we can assume that $\alpha_1\alpha_2\alpha_3\not=0$.

Now, let us show that there is $\tilde{\lambda}=\lambda_1+i\lambda_2\in\comp$ such that $QL_{\tilde{\lambda}}\cap QW(A,B)=\emptyset$, where 
$L_{\Tilde{\lambda}} :=  \{(z,w)\in \comp^2: \Tilde{\lambda} z=w\}=\{(z_1,z_2,w_1,w_2)\in \mathbb{R}^4 : w_1=\lambda_1z_1-\lambda_2z_2, w_2=\lambda_1z_2+\lambda_2z_1 \}$.
Indeed, $$QL_{\Tilde{\lambda}}=\{\Big(\frac{z_1+\lambda_1z_1-\lambda_2z_2}{\sqrt{2}}, z_2,\lambda_1z_2+\lambda_2z_1\Big):z_1,z_2\in \mathbb{R}\}=$$
$$=\{\Big(\big(\frac{1+\lambda_1}{\sqrt{2}\lambda_2}t-\frac{\lambda_1+\lambda_1^2+\lambda_2^2}{\sqrt{2}\lambda_2}z \big), z,t\Big):z,t\in \mathbb{R}\}=\{\big(pt-qz, z,t\big):z,t\in \mathbb{R}\},$$
where $p=\frac{1+\lambda_1}{\sqrt{2}\lambda_2}$ and $q=\big(p+\frac{1}{2p}\big)\lambda_1+\frac{1}{2p}.$

Thus an appropriate choice of $\Tilde{\lambda}$ gives us an arbitrary nonzero numbers $p,q\in\mathbb{R}$.
Hence we can choose $\Tilde{\lambda}$ such that $f(QL_{\Tilde{\lambda}})=0$, so $QL_{\Tilde{\lambda}}\cap QW(A,B)=\emptyset$. And finally, $L_{\Tilde{\lambda}}\cap W(A,B)=\emptyset$. Hence $-\Tilde{\lambda}\notin W(P(\lambda))$.
Thus $W(P(\lambda))\not=\comp$, which is a contradiction.

Therefore, $0\in QW(A,B)$, so there is a vector $v$ such that $v^*\big(\frac{R_1+R_2}{\sqrt{2}}\big)v=0, v^*J_1v=0,  v^*J_2v=0.$ However, $v^*R_1v\geq 0, v^*R_2v\geq 0$. So $v$ is a common isotropic vector.

\end{proof}

\begin{remark}
In Theorem \ref{openMMW} is not enough to assume that only one of the matrices $R_1,R_2$ is nonnegative. 

In Theorem 4.2 \cite{LiPoon1999}, the authors show that for any matrices $A_1,A_2,A_3$ such that $Id,A_1,A_2,A_3$ are linearly independent there exists a matrix $X\in \comp^{n\times 2}$
such that $X^*X=Id_2$ and $W(A_1,A_2,A_3,XX^*)$ is not convex. Moreover, by the proof of Theorem 4.1 \cite{LiPoon1999}, it could be assumed that 
$$ X^*A_1X= \begin{bmatrix}
    0 & 1 \\
    1 & 0
\end{bmatrix}, \quad 
X^*A_2X= \begin{bmatrix}
    0 & i \\
    -i & 0
\end{bmatrix}.
$$
\end{remark}
Thus $W(X^*A_1X,X^*A_2X,X^*A_3X)$ is an ellipsoid with $0$ in its interior. By Lemma 4.1 \cite{Li1988} $(a,b,c,0)\in W(A_1,A_2,A_3,Id-XX^*)$ if and only if $(a,b,c)\in W(X^*A_1X,X^*A_2X,X^*A_3X)$  (see the proof of Theorem 4.2 \cite{LiPoon1999}).

Hence, for $R_1=Id-XX^*$, $R_2=A_1$, $J_1=iA_2$ and $J_2=iA_3$ there is
$$0\not\in W(R_1+J_1,R_2+J_2), \quad 0 \in \text{conv}W(R_1+J_1,R_2+J_2), \quad R_1\geq 0.$$
Thus by Theorem \ref{convIso} $W(R_1+J_1,R_2+J_2)=\comp$ and $R_1+J_1,R_2+J_2$ have no common isotropic vector.

\section{Linear pencils with Hermitian coefficients}\label{self-a}

Theorem 4.1 \cite{Li-Rodman} characterizes the numerical range of linear pencil with Hermitian coefficients, i.e. matrix polynomial $P(\lambda)=A\lambda -B$, where $A$ and $B$ are Hermitian matrices. Unfortunately, as it was observed in \cite{Bebiano}, there is an error in Theorem 4.1(e) \cite{Li-Rodman}. 

If $A\in \comp^{N\times N}$ and $B\in\comp^{N\times N}$ have common isotropic vector, then  $W(P(\lambda))=\comp$. Otherwise, by \cite{THOMPSON1976}, there is a non-singular matrix $X\in\comp^{N\times N}$ such that  $$X^*AX=I_n\oplus-I_m\oplus0_k, \quad X^*BX=\text{diag}(a_1,\ldots,a_n,b_1,\ldots,b_m,c_1\ldots,c_k),$$ for some $a_1,\ldots,a_n,b_1,\ldots,b_m,c_1\ldots,c_k \in \mathbb{R}$. Therefore, to characterize $W(A\lambda-B)$ it is enough to consider such diagonal matrices. Moreover, since $W(A\lambda-B)=W(-(A\lambda-B))$ we can assume that $n\geq m$.

\begin{example}\label{forbebiano}
Let $$A=\begin{bmatrix}
    1 & 0 \\
    0 & -1
\end{bmatrix}, \quad
 B=\begin{bmatrix}
     2 & 0 \\
     0 & -1
\end{bmatrix}.$$
Then $W(A\lambda-B)=\mathbb{R}\setminus(1,2)$. Indeed, the equation $\lambda x^*Ax-x^*Bx=0$ give us $$\lambda |x_1|^2-\lambda |x_2|^2-2|x_1|^2+|x_2|^2=0,$$ where $x=[x_1,x_2]^T$.
So $$\lambda=\frac{2 |x_1|^2-|x_2|^2}{|x_1|^2-|x_2|^2}=1+\frac{|x_1|^2}{|x_1|^2-|x_2|^2}=2+\frac{|x_2|^2}{|x_1|^2-|x_2|^2}.$$
Thus $\lambda\geq2$ for $|x_1|>|x_2|$ and $\lambda\leq 1$ for $|x_1|<|x_2|$.
\end{example}

Theorem 4.1(e) \cite{Li-Rodman} claims that in the above example $W(A\lambda-B)$ should be $\mathbb{R}$.
The authors of \cite{Bebiano} 
improve Theorem 4.1(e) and say that  $W(A\lambda-B)=\mathbb{R}$ or $W(A\lambda-B)=\mathbb{R} \setminus (a,b)$, for some $a<b$.
However, their reasoning is still erroneous. They claim that the case $W(\lambda A-B)=\mathbb{R}$ may occur for some indefinite $A,B$, which is not true (see Corollary \ref{CorForBebiano}). In fact, they consider the cases which they earlier exclude as matrices with common isotropic vector.

In the following proposition, we improve a full characterisation of the numerical range of Hermitian pencils given in \cite{Li-Rodman}.

\begin{theorem}\label{improvedLiRodman}
Let $P(\lambda)=A\lambda-B$ be a matrix pencil with Hermitian coefficients $A=I_n\oplus-I_m\oplus0_k$, $B=B_1\oplus B_2\oplus B_3=\text{diag}(a_1,\ldots,a_n,b_1,\ldots,b_m,c_1\ldots,c_k)$, where $n\geq m$. Suppose that $A$ and $B$ do not have a common non-zero isotropic vector. Then exactly one of the following cases holds:
\begin{enumerate}[label=\alph*)]{
    \item $A$ is positive definite (that is $A=I$) and $W(P(\lambda))=W(B)$. In particular, $W(P(\lambda))$ is a singleton if and only if $B$ is a scalar matrix; $W(P(\lambda))$ is a positive (nonnegative) line segment if and only if $B$ is positive (semi)definite.
    \item $A$ is a singular positive semidefinite matrix; $W(P(\lambda))=[\alpha,\infty)$ if $B_3$ is positive definite, and $W(P(\lambda))=(-\infty,\beta]$ if $B_3$ is negative definite, where $W(B_1)=[\alpha, \beta]$. In this case, $B_3$ must either be positive definite or negative definite because $A$ and $B$ have no common isotropic vector.
    \item $A$ is indefinite and $B$ is positive (negative) definite and $W(P(\lambda))=\RR\setminus(\alpha,\beta)$, where $W(B^{-\frac{1}{2}}AB^{-\frac{1}{2}})=[\frac{1}{\alpha}, \frac{1}{\beta}]$ ($W((-B^{-\frac{1}{2}})(-A)(-B^{-\frac{1}{2}}))=[\frac{1}{\alpha}, \frac{1}{\beta}]$). In this case, $W(P(\lambda))$ is the union of two disjoint unbounded intervals and $0\not \in W(P(\lambda))$. 
    \item $A$ is indefinite and $B$ is a singular positive (negative) semidefinite matrix, and $W(P(\lambda))=\{\mu^{-1}:\mu \in W(B\lambda - A), \mu \neq 0\}\cup \{0\}$ ($W(P(\lambda))=\{\mu^{-1}:\mu \in W((-B)\lambda - (-A)), \mu \neq 0\}\cup \{0\}$). In this case, $W(P(\lambda))$ is the union of two disjoint unbounded intervals and $0\in W(P(\lambda))$.
    \item Both $A$ and $B$ are indefinite and $W(P(\lambda))=\RR\setminus (\alpha,\beta)$ for some $\alpha <\beta$. More precisely, let $W(B_1)=[\alpha_1,\beta_1]$, $W(B_2)=[\alpha_2,\beta_2]$. If $k>0$ or $\alpha_1+\alpha_2>0$, then $\alpha=-\alpha_2$, $\beta=\alpha_1$. Otherwise, $\alpha=\beta_1$, $\beta=-\beta_2$. 
}\end{enumerate}

\end{theorem}
\begin{proof}
Proofs of (a)--(d) can be found in \cite{Li-Rodman}, so we will concentrate on case (e). 

\noindent Without loss of generality we may assume that $a_1\leq a_2\ldots\leq a_n$, $b_1\leq b_2\ldots\leq b_m$, $c_1\leq c_2\ldots\leq c_k$. Obviously, if for some $i,j$ equality $a_i+b_j=0$ holds, then thr vector $$(0,\ldots,0,1,0,\ldots,0,1,0,\ldots,0), \quad \text{units are on the $i$-th and $n+j$-th places}$$ would be an isotropic vector for $A$ and $B$, thus such equalities cannot hold. 

\noindent \textbf{Case 1:} $k\neq0$. 

\noindent We may assume that $c_1>0$ as all $c_i$ must have the same sign (otherwise there will be an isotropic vector). Then equality $x^*(\lambda A-B)x=0$ becomes 
\begin{equation}\label{lambdaexpr1}
\lambda = \dfrac{\sum_{i=1}^n{|x_i|^2a_i}+\sum_{j=1}^m{|y_j|^2b_j}+\sum_{l=1}^k{|z_l|^2c_l}}{\sum_{i=1}^n|x_i|^2-\sum_{j=1}^m|y_j|^2},
\end{equation}
\noindent where $x=(x_1,\ldots,x_n,y_1,\ldots,y_m,z_1\ldots,z_k)^T$. 
\noindent It is easy to see that $a_1+b_1>0$ as otherwise vector $(1,0,\ldots,0,1,0,\ldots,0,\sqrt{\frac{-a_1-b_1}{c_1}},0,\ldots,0)$ would be an isotropic one. Then we can rewrite \eqref{lambdaexpr1} as 
\begin{equation*}
\lambda = a_1-\dfrac{\sum_{i=1}^n{|x_i|^2(a_i-a_1)}+\sum_{j=1}^m{|y_j|^2(b_j+a_1)}+\sum_{l=1}^k{|z_l|^2c_l}}{\sum_{j=1}^m|y_j|^2-\sum_{i=1}^n|x_i|^2}.
\end{equation*}
However, that means that $\lambda\not\in (-b_1,a_1)$ as $$\sum_{i=1}^n{|x_i|^2(a_i-a_1)}+\sum_{j=1}^m{|y_j|^2(b_j+a_1)}+\sum_{l=1}^k{|z_l|^2c_l}\geq(a_1+b_1)(\sum_{j=1}^m|y_j|^2-\sum_{i=1}^n|x_i|^2).$$
Furthermore, $\lambda$ can achieve any value from $\RR\setminus (-b_1,a_1)$ as $a_1+t$, for $t>0$, can be achieved by the vector $$\Big(1,0,\ldots,0,\dfrac{t}{a_1+b_1},0,\ldots,0\Big), \quad \text{where $\dfrac{t}{a_1+b_1}$ is on $(n+1)^{\text{st}}$ position}  $$
and $-b_1-t$, for $t>0$, can be achieved by the vector $$\Big(\dfrac{t}{t+a_1+b_1},0,\ldots,0,1,0,\ldots,0\Big), \quad \text{where 1 is on $(n+1)^{\text{st}}$ position}  $$

\noindent \textbf{Case 2:} $k=0$. 

\noindent First of all, we note that \eqref{lambdaexpr1} becomes 
\begin{equation} \label{lambdaexpr2}
\lambda = \dfrac{\sum_{i=1}^n{|x_i|^2a_i}+\sum_{j=1}^m{|y_j|^2b_j}}{\sum_{i=1}^n|x_i|^2-\sum^m_{j=1}|y_j|^2},
\end{equation} in this case. Furthermore, we claim that if $A$ and $B$ have no common isotropic vector, then either $a_1+b_1>0$ or $a_n+b_m<0$. Note that if $n=m=1$, then the claim trivially holds, so we proceed with the assumption $n\geq 2$, $m\geq 1$. If $a_1+b_1<0$ and $a_n+b_m>0$, then $$x=\bigg(\sqrt{\frac{-a_n-b_m}{a_1+b_m}},0,\ldots,0,1,0,\ldots,0,\sqrt{\frac{a_1-a_n}{a_1+b_m}}\bigg), \quad \text{$1$ is on the $n^{\text{th}}$ position}$$ would be an isotropic vector if $a_1+b_m<0$ or $$x=\bigg(\sqrt{1-\frac{a_1+b_1}{a_1+b_m}},0,\ldots,0,1,0,\ldots,0,\sqrt{-\frac{a_1+b_1}{a_1+b_m}}\bigg), \quad \text{$1$ is on the $(n+1)^{\text{st}}$ position}$$ 
would be an isotropic vector if $a_1+b_m>0$. Note that in this case $m\geq 2$, as $a_1+b_1$ is already assumed to be negative.

\noindent If $a_1+b_1>0$, then the same argument as in the first case shows that $W(P(\lambda))=\RR\setminus (-b_1,a_1)$. 

\noindent If $a_n+b_m<0$, rewrite \eqref{lambdaexpr2} as 
\begin{equation}
\lambda = a_n+\dfrac{\sum^n_{i=1}{|x_i|^2(a_n-a_i)}+\sum^m_{j=1}{|y_j|^2(-a_n-b_j)}}{\sum^m_{j=1}|y_j|^2-\sum_{i=1}^n|x_i|^2},
\end{equation}
and note that $$\sum_{i=1}^n{|x_i|^2(a_n-a_i)}+\sum_{j=1}^m{|y_j|^2(-a_n-b_j)}\geq (-a_n-b_m)(\sum^m_{j=1}|y_j|^2-\sum_{i=1}^n|x_i|^2),$$ so $\lambda\not\in (a_n,-b_m)$. All other values can be reached as in the first case and that ends the proof.
\end{proof}

As we mentioned at the beginning of this section, all cases considered in the above theorem characterize the numerical range of any linear pencil $A\lambda-B$, where $A,B$ are Hermitian and do not have a common isotropic vector. Since in all cases $\mathbb{R}\not\subset W(A\lambda-B)$, one can conclude Corollary \ref{CorForBebiano} from Theorem \ref{improvedLiRodman}

\section{Matrix polynomials with Hermitian coefficients}\label{last}

The following example explains that for matrix polynomials with degree larger than $1$, the property $W(P(\lambda))=\comp$ do not implies existence of common isotropic vector even if the matrices are Hermitian and diagonal.

\begin{example}\label{VadymEx}
The matrix polynomial $$P(\lambda)=\lambda^2  \left( \begin{array}{ccc}
1 & 0 & 0 \\
0 & -1 & 0 \\
0 & 0 & 0 \end{array} \right) +
\lambda  \left( \begin{array}{ccc}
4 & 0 & 0 \\
0 & -3 & 0 \\
0 & 0 & -4 \end{array} \right) +
 \left( \begin{array}{ccc}
-1 & 0 & 0 \\
0 & 2 & 0 \\
0 & 0 & -9 \end{array} \right)$$
\noindent satisfies $W(P(\lambda))=\CC$ and its coefficients do not have a common isotropic vector. 
\end{example}
\begin{proof}
\noindent In order to show the first part of the claim, it would be sufficient to show that for any $a,b\in\RR$ there exists a non-zero vector $v\in\CC^3$ such that $\lambda=a+bi$ is a solution to the equation $v^*P(\lambda)v=0$. Let $v=(\sqrt{x},\sqrt{y},\sqrt{z})$, where $x,y,z\in\RR_{\geq 0}$. Then the equation becomes $$\lambda^2(x-y)+\lambda (4x-3y-4z)+(-x+2y-9z)=0.$$ 
\noindent The quadratic formula gives us solutions $$\lambda_{1,2}=\dfrac{-4x+3y+4z\pm \sqrt{(4x-3y-4z)^2-4(x-y)(-x+2y-9z)}}{2(x-y)},$$
\noindent so it will be enough to find $x,y,z$ such that $$a=\dfrac{-4x+3y+4z}{2(x-y)}, \quad -b^2=\dfrac{(4x-3y-4z)^2-4(x-y)(-x+2y-9z)}{4(x-y)^2},$$
\noindent which is equivalent to
\begin{align*}
a=\dfrac{-4x+3y+4z}{2(x-y)}&, \quad -b^2=a^2-\dfrac{(-x+2y-9z)}{(x-y)} \\
a+2-\dfrac{2z}{x-y}=-\dfrac{y}{x-y}&, \quad -b^2-a^2-\dfrac{9z}{x-y}-1=-\dfrac{y}{(x-y)} \\
\dfrac{z}{x-y}=\dfrac{-b^2-a^2-a-3}{7}&, \quad \dfrac{y}{x-y}=\dfrac{-2b^2-2a^2-9a-20}{7}.
\end{align*}
It is easy to see that $\dfrac{-b^2-a^2-a-3}{7}<0$ and $\dfrac{-2b^2-2a^2-9a-20}{7}<-1$, so for every $y>0$ one can find $x>0$, which satisfies the second equation. However, for every such pair $x-y<0$, thus there exists $z>0$ which satisfies the first equation, showing that $W(P(\lambda))=\CC$. 

In order to see that the coefficients do not have a common isotropic vector, suppose that $v=(z_1,z_2,z_3)$ would be a such one. Then we get $$|z_1|^2-|z_2|^2=0, \quad 4|z_1|^2-3|z_2|^2-4|z_3|^2=0, \quad -|z_1|^2+2|z_2|^2-9|z_3|^2=0,$$
\noindent but this system's only solution is $z_1=z_2=z_3=0$, thus proving the rest of the claim.
\end{proof}

If we assume that matrices $A$, $B$, $C$ are semidefinite, then no such example exist. 

\begin{proposition}
If $A,B,C$ are semidefinite matrices such that $W(P(\lambda))=\comp$, where $P(\lambda)=A\lambda^2+B\lambda+C$, then $A,B,C$ have a common isotropic vector.
\end{proposition}
\begin{proof}
Let us assume that $A,B,C$ do not have a common isotropic vector.

Since $W(P(\lambda))=W(-P(\lambda))$, it is enough to consider four cases:\\
(1) $A\geq 0, B\geq 0, C\geq0$\\
Then $\mathbb{R}_+\not\subset W(P(\lambda))$, because for $\lambda>0$ we have $$\lambda^2x^*Ax+x^*Cx \geq 0 \geq -\lambda x^*Bx,$$
and since $A,B,C$ has no common isotropic vector at least one inequality is strict.\\
(2) $A\geq 0, B\leq 0, C\geq0$

Then $\mathbb{R}_-\not\subset W(P(\lambda))$, because for $\lambda<0$ we have $$\lambda^2x^*Ax+x^*Cx \geq 0 \geq -\lambda x^*Bx,$$
and since $A,B,C$ has no common isotropic vector at least one inequality is strict.\\
(3) $A\geq 0, B\geq 0, C\leq0$ and (4) $A\geq 0, B\leq 0, C\leq0$

Since $$(x^*Bx)^2\geq 0 \geq x^*Ax\cdot x^*Cx,$$
all solutions of the equation $$\lambda^2x^*Ax+ \lambda x^*Bx+x^*Cx=0,$$ are real. Thus $W(P(\lambda))\subset \mathbb{R}$.
\end{proof}

\section*{Acknowledgment}
The authors are sincerely grateful to Michał Wojtylak for bringing their attention to this subject and for his all inspiring comments.

Second author acknowledges the financial support by the JU (POB DigiWorld) grant, decision No. PSP: U1U/P06/NO/02.12

\bibliographystyle{alpha}
\bibliography{sample}

\newcommand{\etalchar}[1]{$^{#1}$}
\begin{thebibliography}{MMMM06b}

\bibitem[AYP79]{au1979remark}
Yik-Hoi Au-Yeung and Yiu-Tung Poon.
\newblock A remark on the convexity and positive definiteness concerning
  hermitian matrices.
\newblock {\em Southeast Asian Bull. Math}, 3(2):85--92, 1979.

\bibitem[AYT83]{extHT}
Yik-Hoi Au-Yeung and Nam-Kiu Tsing.
\newblock An extension of the hausdorff-toeplitz theorem on the numerical
  range.
\newblock {\em Proceedings of the American Mathematical Society},
  89(2):215--218, 1983.

\bibitem[BdPE17]{Beb1}
Nat{\'a}lia Bebiano, Jo{\~a}o da~Provid{\^e}ncia, and Fatemeh Esmaeili.
\newblock The characteristic polynomial of linear pencils of small size and the
  numerical range.
\newblock In Nat{\'a}lia Bebiano, editor, {\em Applied and Computational Matrix
  Analysis}, pages 181--197, Cham, 2017. Springer International Publishing.

\bibitem[BdPNdP17]{Bebiano}
Nat{\'a}lia Bebiano, Jo{\~a}o da~Provid{\^e}ncia, Ana Nata, and Jo{\~a}o~P.
  da~Provid{\^e}ncia.
\newblock Fields of values of linear pencils and spectral inclusion regions.
\newblock In Nat{\'a}lia Bebiano, editor, {\em Applied and Computational Matrix
  Analysis}, pages 165--179, Cham, 2017. Springer International Publishing.

\bibitem[BHM{\etalchar{+}}13]{NLEVP}
Timo Betcke, Nicholas~J. Higham, Volker Mehrmann, Christian Schr\"{o}der, and
  Fran\c{c}oise Tisseur.
\newblock Nlevp: A collection of nonlinear eigenvalue problems.
\newblock {\em ACM Trans. Math. Softw.}, 39(2), feb 2013.

\bibitem[CN02]{CHIEN200269}
Mao-Ting Chien and Hiroshi Nakazato.
\newblock The numerical range of linear pencils of 2-by-2 matrices.
\newblock {\em Linear Algebra and its Applications}, 341(1):69--100, 2002.
\newblock Special issue dedicated to Professor T. Ando.

\bibitem[CNP02]{CHIEN2002205}
Mao-Ting Chien, Hiroshi Nakazato, and Panayiotis~J. Psarrakos.
\newblock Point equation of the boundary of the numerical range of a matrix
  polynomial.
\newblock {\em Linear Algebra and its Applications}, 347(1):205--217, 2002.

\bibitem[Dav71]{davis_1971}
Chandler Davis.
\newblock The {T}oeplitz-{H}ausdorff theorem explained.
\newblock {\em Canadian Mathematical Bulletin}, 14(2):245--246, 1971.

\bibitem[GLR09]{MR3396732}
Israel Gohberg, Peter Lancaster, and Leiba Rodman.
\newblock {\em Matrix polynomials}, volume~58 of {\em Classics in Applied
  Mathematics}.
\newblock Society for Industrial and Applied Mathematics (SIAM), Philadelphia,
  PA, 2009.
\newblock Reprint of the 1982 original [ MR0662418].

\bibitem[IW16]{ITO2016683}
Naoharu Ito and Harald~K. Wimmer.
\newblock Self-inversive {H}ilbert space operator polynomials with spectrum on
  the unit circle.
\newblock {\em Journal of Mathematical Analysis and Applications},
  436(2):683--691, 2016.

\bibitem[KL98]{KRUPNIK199845}
Ilya Krupnik and Peter Lancaster.
\newblock Linearizations, realization, and scalar products for regular matrix
  polynomials.
\newblock {\em Linear Algebra and its Applications}, 272(1):45--57, 1998.

\bibitem[Lan08]{LinLancaster}
Peter Lancaster.
\newblock Linearization of regular matrix polynomials.
\newblock {\em Electron. J. Linear Algebra}, 17:21--27, 2008.

\bibitem[Li88]{Li1988}
Chi-Kwong Li.
\newblock Matrices with some extremal properties.
\newblock {\em Linear Algebra Appl.}, 101:255--267, 1988.

\bibitem[LP99]{LiPoon1999}
Chi-Kwong Li and Yiu-Tung Poon.
\newblock Convexity of the joint numerical range.
\newblock {\em SIAM J. Matrix Anal. Appl.}, 21(2):668--678, 1999.

\bibitem[LR94]{Li-Rodman}
Chi-Kwong Li and Leiba Rodman.
\newblock Numerical range of matrix polynomials.
\newblock {\em SIAM J. Matrix Anal. Appl.}, 15(4):1256--1265, 1994.

\bibitem[MMMM06a]{Linearizations2}
D.~Steven Mackey, Niloufer Mackey, Christian Mehl, and Volker Mehrmann.
\newblock Structured polynomial eigenvalue problems: Good vibrations from good
  linearizations.
\newblock {\em SIAM Journal on Matrix Analysis and Applications},
  28(4):1029--1051, 2006.

\bibitem[MMMM06b]{Linearizations}
D.~Steven Mackey, Niloufer Mackey, Christian Mehl, and Volker Mehrmann.
\newblock Vector spaces of linearizations for matrix polynomials.
\newblock {\em SIAM Journal on Matrix Analysis and Applications},
  28(4):971--1004, 2006.

\bibitem[MMW18]{MMWsimax2018}
Christian Mehl, Volker Mehrmann, and Michal Wojtylak.
\newblock Linear algebra properties of dissipative hamiltonian descriptor
  systems.
\newblock {\em SIAM Journal on Matrix Analysis and Applications},
  39(3):1489--1519, 2018.

\bibitem[MMW21a]{MMW2021}
Christian Mehl, Volker Mehrmann, and Michal Wojtylak.
\newblock Distance problems for dissipative {H}amiltonian systems and related
  matrix polynomials.
\newblock {\em Linear Algebra Appl.}, 623:335--366, 2021.

\bibitem[MMW21b]{mehl2021matrix}
Christian Mehl, Volker Mehrmann, and Michal Wojtylak.
\newblock Matrix pencils with coefficients that have positive semidefinite
  hermitian part, 2021.

\bibitem[MP96]{MAROULAS199641}
John Maroulas and Panayiotis~J. Psarrakos.
\newblock Geometrical properties of numerical range of matrix polynomials.
\newblock {\em Computers \& Mathematics with Applications}, 31(4):41--47, 1996.
\newblock Selected Topics in Numerical Methods.

\bibitem[MP97]{MAROULAS1997101}
John Maroulas and Panayiotis~J. Psarrakos.
\newblock The boundary of the numerical range of matrix polynomials.
\newblock {\em Linear Algebra and its Applications}, 267:101--111, 1997.

\bibitem[MP98]{MAROULAS199897}
John Maroulas and Panayiotis~J. Psarrakos.
\newblock On the connectedness of numerical range of matrix polynomials.
\newblock {\em Linear Algebra and its Applications}, 280(2):97--108, 1998.

\bibitem[MW02]{MehrmannWatkins}
Volker Mehrmann and David Watkins.
\newblock Polynomial eigenvalue problems with {H}amiltonian structure.
\newblock {\em Electron. Trans. Numer. Anal.}, 13:106--118, 2002.

\bibitem[NP01]{NAKAZATO2001105}
Hiroshi Nakazato and Panayiotis~J. Psarrakos.
\newblock On the shape of numerical range of matrix polynomials.
\newblock {\em Linear Algebra and its Applications}, 338(1):105--123, 2001.

\bibitem[Psa00]{PSARRAKOS2000127}
Panayiotis~J. Psarrakos.
\newblock Numerical range of linear pencils.
\newblock {\em Linear Algebra and its Applications}, 317(1):127--141, 2000.

\bibitem[PT05]{approxBoundary}
Panayiotis~J. Psarrakos and Charalampos Tsitouras.
\newblock Numerical approximation of the boundary of numerical range of matrix
  polynomials.
\newblock {\em Applied Numerical Analysis \& Computational Mathematics},
  2(1):126--133, 2005.

\bibitem[RS11]{Ratio}
Leiba Rodman and Ilya~M. Spitkovsky.
\newblock Ratio numerical ranges of operators.
\newblock {\em Integr. Equ. Oper. Theory}, 71(2):245--257, 2011.

\bibitem[SW22]{SzymanskiWojtylak}
Oskar~Jakub Szymański and Michał Wojtylak.
\newblock Stability of matrix polynomials in one and several variables, 2022.

\bibitem[Tho76]{THOMPSON1976}
Robert~C. Thompson.
\newblock The characteristic polynomial of a principal subpencil of a hermitian
  matrix pencil.
\newblock {\em Linear Algebra and its Applications}, 14(2):135--177, 1976.

\end{thebibliography}

\address{Faculty of Mathematics and Computer Science\\
   Jagiellonian University\\
   \L ojasiewicza 6\\
   30-348 Krak\'ow\\
   Poland}
\newline
\email{vadym.koval@student.uj.edu.pl,\\patryk.pagacz@uj.edu.pl }

\end{document}